\documentclass[12pt,twoside]{amsart}
\usepackage{graphicx}
\usepackage{tikz}
\usetikzlibrary{decorations.markings}
\usetikzlibrary{arrows,shapes,positioning}
\usetikzlibrary{topaths}
\tikzstyle{vertex}=[circle, draw, inner sep=0pt, minimum size=4pt]
\usetikzlibrary{positioning}
\newcommand{\vertex}{\node[vertex]}

\usepackage{amssymb,amscd}
\usepackage{mathrsfs}
\usepackage{array,float}
\usepackage[all]{xy}
\usepackage{enumerate}

\theoremstyle{plain}
\newtheorem{thm}{Theorem}[section]
\newtheorem{thmx}{Theorem}

\newtheorem{corx}[thmx]{Corollary}

\newtheorem{lem}[thm]{Lemma}
\newtheorem{pro}[thm]{Proposition}

\newtheorem{conjecture}[thm]{Conjecture}

\newtheorem{question}[thm]{Question}

\theoremstyle{remark}
\newtheorem{rem}[thm]{Remark}

\theoremstyle{definition} 

\newtheorem{example}[thm]{Example}

\numberwithin{equation}{section}

\newcommand{\ackn}{  \noindent{\sc Acknowledgement }\hspace{5pt} }

\renewcommand{\phi}{\varphi}

\begin{document}

\author{Ilir Snopce}
\address{Universidade Federal do Rio de Janeiro\\
  Instituto de Matem\'atica \\
  21941-909 Rio de Janeiro, RJ \\ Brazil }
\email{ilir@im.ufrj.br}

\author{Slobodan Tanushevski} \address{Universidade Federal Fluminense\\
  Instituto de Matem\'atica e Estat\'istica \\
  24210-201, Niter\'oi, RJ \\ Brazil }
\email{tanusevski1@gmail.com}

\author{Pavel Zalesskii}
\address{University of Bras\'ilia\\
  Department of  Mathematics \\
  70910-9000, Bras\'ilia \\ Brazil }
\email{pz@mat.unb.br}

\thanks{}

\title[Retracts of free groups and a question of Bergman] {Retracts of free groups and a question of Bergman}

\begin{abstract} 
Let $F_n$ be a free group of finite rank $n \geq 2$.
We prove that if $H$ is a subgroup of $F_n$ with $\textrm{rk}(H)=2$ and $R$ is a retract of $F_n$, then $H \cap R$ is a retract of $H$.
However, for every $m \geq 3$ and every $1 \leq k \leq n-1$, there exist a subgroup $H$ of $F_n$ of rank $m$ and a retract $R$ of $F_n$ of rank $k$ such that
$H \cap R$ is not a retract of $H$. This gives a complete answer to a question of Bergman.

Furthermore, we provide positive evidence for the inertia conjecture of Dicks and Ventura. More precisely, we prove that $\textrm{rk}(H \cap \textrm{Fix}(S)) \leq \textrm{rk}(H)$ 
for every family $S$ of endomorphisms of $F_n$ and every subgroup $H$ of $F_n$ with $\textrm{rk}(H) \leq 3$.

\end{abstract}


\maketitle
\section{Introduction}
Throughout,  $F_n$ denotes a free group of finite rank $n \geq 2$. 
A subgroup $R \leq F_n$ is a  retract of $F_n$ if there exists a homomorphism $r:F_n \to R$ (called a retraction) that restricts to the identity on $R$.
In 1999, Bergman proved the following 
\begin{thm}[Bergman, \cite{Bergman}] 
\label{intersection of retracts}
The intersection of any family of retracts of $F_n$ is a retract of $F_n$. 
\end{thm}
In the same paper, he raised the following 
\begin{question}[Bergman, \cite{Bergman}]
\label{Bergman}
Let $R$ be a retract of $F_n$. Is $H \cap R$ a retract of $H$ for every finitely generated subgroup $H$ of $F_n$?
\end{question}

The same question also appears in \cite[Problem~F11]{Baumslag}, \cite[Problem~17.19]{Khukhro-Mazurov} and \cite{Ventura}.  In addition to being important on its own right, another reason for the sustained interest in Bergman's question is due to its connection to the study of fixed subgroups of endomorphisms of free groups.

For a given family $S$ of endomorphisms of $F_n$, let 
\[\textrm{Fix}(S)=\{w \in F_n \mid \varphi(w)=w \text{ for every } \varphi \in S\}\] 
denote the fixed subgroup of $S$. In the seminal paper \cite{Bestvina-Handel}, Bestvina and Handel proved that $\textrm{rk}(\textrm{Fix}(\alpha)) \leq n$ for every automorphism $\alpha$ of $F_n$.
By an elementary algebraic argument, Imrich and Turner \cite{Imrich-Turner} extended this result to all endomorphisms of $F_n$. 
In the monograph \cite{Dicks}, Dicks and Ventura introduced the concept of inertia of subgroups of free groups: A subgroup $H$ of $F_n$ is inert if $\textrm{rk}(K \cap H) \leq \textrm{rk}(K)$ for every subgroup $K$ of $F_n$.
After reformulating (in a more algebraic language) and extending the Bestvina-Handel theory, Dicks and Ventura proved that $\textrm{Fix}(S)$ is inert (in particular, $\textrm{rk}(\textrm{Fix}(S)) \leq n$) for every family $S$ of injective endomorphisms of $F_n$. 
Furthermore, they conjectured that $\textrm{Fix}(S)$ is inert for an arbitrary family $S$ of endomorphisms of $F_n$. 
In \cite{Bergman}, Bergman provided evidence for the Dicks-Ventura conjecture by proving the pinnacle result on the ranks of fixed subgroups of endomorphisms of free groups:
$\textrm{rk}(\textrm{Fix}(S)) \leq n$ for every family $S$ of endomorphisms of $F_n$.

By an argument due to Turner \cite{Turner}, the Dicks-Ventura conjecture is equivalent to the following

\begin{conjecture}[Dicks-Ventura, \cite{Dicks}]
\label{Dicks-Ventura}
Every retract of $F_n$ is inert. 
\end{conjecture}

Since the rank of a retract of $F_n$ is at most $n$ (in fact, every proper retract of $F_n$ has rank smaller than $n$), 
it follows that a positive answer to Bergman's question would imply the Dicks-Ventura conjecture.
(For a comprehensive history of the theory of fixed subgroups of endomorphisms of free groups, we refer the reader to the survey paper \cite{Ventura}.)

\medskip

In this paper, we answer the question of Bergman.
 
\begin{thmx}
\label{main thm}
\begin{enumerate}[(i)]
\item Let $H$ be a subgroup of $F_n$ of rank two, and let $R$ be a retract of $F_n$. Then $H \cap R$ is a retract of $H$.
\item For every $m \geq 3$ and every $1 \leq k \leq n-1$, there exist a subgroup $H$ of $F_n$ of rank $m$ and a retract $R$ of $F_n$ of rank $k$ such that 
$H \cap R$ is not a retract of $H$. 
\end{enumerate}
\end{thmx}

In contrast to the negative result of Theorem~\ref{main thm} $(ii)$, our next theorem provides further evidence for the Dicks-Ventura conjecture.

\begin{thmx}
\label{D-V}
Let $R$ be a retract of $F_n$, and let $H$ be a subgroup of $F_n$ such that $\textrm{rk}(H) \leq 3$. Then $\textrm{rk}(H \cap R) \leq \textrm{rk}(H)$. 
\end{thmx}

As a consequence of this theorem, we get the following
  
\begin{corx}
\label{cor D-V}
Let $S$ be a family of endomorphisms of $F_n$, and let $H$ be a subgroup of $F_n$ with $\textrm{rk}(H) \leq 3$. Then $\textrm{rk}(H \cap \textrm{Fix}(S)) \leq \textrm{rk}(H)$. 
\end{corx}

In Section~\ref{rank two}, we prove Theorem~\ref{main thm} $(i)$. The proof uses recent results on test elements of free pro-$p$ groups; the main step is a result reminiscent of the Prime Avoidance Lemma from commutative algebra (see Lemma~\ref{first step}).
After some preliminary `positive' results on visible elements (which, we believe, are of independent interest), in Section~\ref{higher rank}, we complete the proof of Theorem~\ref{main thm}. The arguments in this section have a more geometric flavor.
Section~\ref{Dicks-Ventura} is dedicated to the proof of Theorem~\ref{D-V}; here we use pro-$p$ techniques and  the Hanna Neumann conjecture (proved in 2011 independently by Friedman and Mineyev ).
\section{When $H$ is of rank two}
\label{rank two}

An element $g$ of a group $G$ is called a test element if every endomorphism of $G$ that fixes $g$ is an automorphism.
Let $x_1, x_2, \ldots, x_n$ be a basis of $F_n$; then
\begin{itemize}
\item the commutator $[x_1, x_2]$ is a test element of $F_2$ (Nielsen, \cite{Nielsen});
\item $[x_1, x_2][x_3,x_4]\cdots [x_{2m-1}, x_{2m}]$ is a test element of $F_{2m}$ (Zieschang, \cite{Zieschang1});
\item every higher commutator of weight $n$ (with arbitrary disposition of commutator brackets) involving all $n$ letters $x_1,x_2, \ldots, x_n$ is a test element of $F_n$ (Rips,  \cite{Rips});
\item $x_1^{k_1}x_2^{k_2}\cdots x_n^{k_n}$  is a test element of $F_n$ if and only if $k_i \neq 0$ for all $1 \leq i \leq n$ and $\gcd(k_1, \ldots, k_n) \neq 1$ (Turner, \cite{Turner});
\item the set of test elements of $F_n$ forms a net in the Cayley graph of $F_n$ (\cite{SnoTan1}).
\end{itemize} 

\medskip

There is a close connection between retracts and test elements of free groups.

\begin{thm}[Turner, \cite{Turner}]
\label{Turner}
An element $w \in F_n$ is a test element if and only if it does not belong to a proper retract of $F_n$.
\end{thm}   

Let $p$ be a prime. A pro-$p$ group is a compact Hausdorff topological group whose open subgroups form a base for the neighborhoods of the identity and every open normal subgroup has index a power of $p$.
Equivalently, a pro-$p$ group is an inverse limit of an inverse system of finite $p$-groups.

Given  a discrete group $G$, the pro-$p$ completion $ \widehat{G}_{p}$  of $G$ is defined as the inverse limit of the (obvious) inverse system formed by the finite quotients $G/N$, where $N$ runs through the normal subgroups of $G$ of index a finite power of $p$. 
There is a natural homomorphism $\jmath_{p}:G \to \widehat{G}_{p}$ determined by the projections $G \to G/N$. If $G$ is residually finite-$p$, then $\jmath_p$ is an embedding and we identify $\jmath_{p}(G)$ with $G$. 
The pro-$p$ completion $\widehat{F}_{n,p}$ of $F_n$ is a free pro-$p$ group. Since free groups are residually finite-$p$, $F_n \leq \widehat{F}_{n,p}$ and every basis of $F_n$ is a basis of $\widehat{F}_{n,p}$ (as a free pro-$p$ group). 
We refer the reader to {\cite[Section~3.2 and Section~3.3]{Ribes1}} for more details on pro-$p$ completions and free pro-$p$ groups.


In \cite{SnoTan2}, test elements in pro-$p$ groups were studied. In particular, the following results were proved.

\begin{thm}
\label{pro-p basic}
\begin{enumerate}[(i)]
\item \cite[Corollary~3.6]{SnoTan2} Let $p$ be a prime. Then $w \in \widehat{F}_{n,p}$ is a test element of $\widehat{F}_{n,p}$ if and only if it is not contained in a proper free factor of $\widehat{F}_{n,p}$.
\item \cite[Corollary~7.2]{SnoTan2} If $w \in F_n$ is a test element of $\widehat{F}_{n,p}$ for some prime $p$, then it is a test element of $F_n$.
\item \cite[Proposition~7.6]{SnoTan2} $w \in F_2$ is a test element of $F_2$ if and only if it is a test element of $\widehat{F}_{2,p}$ for some prime $p$. 
\item \cite[Proposition~5.13 $(b)$ and Proposition~5.10 $(b)$]{SnoTan2} Let $p$ be any prime and let $x_1, x_2$ be a basis of $\widehat{F}_{2,p}$. Then $[x_1, x_2]$ is a test element of $\widehat{F}_{2,p}$.
\end{enumerate} 
\end{thm} 




\begin{rem}
Theorem~\ref{pro-p basic} $(iii)$ does not extend to free groups of higher rank. Indeed, it was proved in \cite{SnoTan3} that for each $n \geq 3$, there are test elements of $F_n$ that are not test elements of $\widehat{F}_{n,p}$ for any prime $p$.
\end{rem}

In the following proposition, we collect a few basic facts on retracts of free pro-$p$ groups that will be used several times in the ensuing arguments.

\begin{pro} 
\label{retracts in pro-p}
Let $p$ be a prime.
\begin{enumerate}[(i)]
\item \cite[Corollary~3.6 and Proposition~7.1]{SnoTan2} If $R$ is a retract of $F_n$, then $\overline{R}$  (the closure of $R$ in $\widehat{F}_{n,p}$) is a free factor of $\widehat{F}_{n,p}$.
\item \cite[Lemma~3.7]{LubotzkyA1} If $H$ is a (topologically) finitely generated closed subgroup of $\widehat{F}_{n,p}$ and $K$ is a free factor of $\widehat{F}_{n,p}$, then
$H \cap K$ is a free factor of $H$.
\end{enumerate}
\end{pro}

The following Lemma will be essential in the proof of Theorem~\ref{main thm} $(i)$, however, it seems to be of independent interest as well.

\begin{lem}
Let $H$ be a subgroup of $F_n$, and let $\{R_i \mid i \in I\}$ be a family of retracts of $F_n$. 
Then the following holds:
\label{first step}
\begin{enumerate}[(a)]
\item  If $x_1, \ldots, x_m$ is a basis of $H$ and $[x_1, \ldots, x_m] \in R_i$ for some $i \in I$, then $H \leq R_i$.
\item  If $H \subseteq \bigcup_{i \in I} R_i$, then $H \leq R_{i}$ for some $i \in I$.
\end{enumerate}
\end{lem}
\begin{proof} $(a)$ Suppose that $R_i$ contains the higher commutator $t=[x_1, \ldots, x_m]$, where 
$x_1, \ldots, x_m$ is a basis of $H$. The proof is by induction on $m$. The case $m=1$ being trivial, we may assume that $m \geq 2$. 

Let $r:F_n \to R_i$ be a retraction; set $s=[x_1, \ldots, x_{m-1}]$, $L=\langle s, x_m \rangle$, $S=\langle r(s), r(x_m) \rangle$, and 
$K=\langle L, S \rangle$. The restriction of $r$ to $K$ is a retraction from $K$ onto $S$. Furthermore, as $t \in R_i$, we have 
$[r(s), r(x_m)]=r(t)=t \neq 1$; thus $S$ has rank two.

Let $p$ be any prime and consider the pro-$p$ completion $\widehat{K}_p$ of $K$. 
Since $\overline{L} $ (the closure of $L$ in $\widehat{K}_p$) is a free pro-$p$ group of rank two with basis $s, x_m$, it follows from
Theorem~\ref{pro-p basic}~$(iv)$ that $t=[s, x_m]$ is a test element of $\overline{L}$.
By Proposition~\ref{retracts in pro-p}, the closure of $S$ in $\widehat{K}_p$ is a free factor of $\widehat{K}_p$ and  $\overline{S} \cap \overline{L}$ is a free factor of $\overline{L}$. 
Since $t \in \overline{S} \cap \overline{L}$ and $t$ is a test element of $\overline{L}$, it follows from Theorem~\ref{pro-p basic}~$(i)$ that 
$\overline{S} \cap \overline{L}=\overline{L}$.  Therefore, $\widehat{K}_p=\overline{S}$, and thus $\widehat{K}_p$ has rank two.
This implies that $K$ also has rank two. Since $S$ and $K$ have the same rank and $S$ is a retract of $K$, we may conclude that $K=S$. 
Hence, $s, x_m \in R_i$, and by applying the induction hypothesis to $s$ and the subgroup $\langle x_1, x_2, \ldots, x_{m-1} \rangle$, we get that 
$x_1, x_2, \ldots, x_{m-1} \in R_i$. Therefore, $H \leq R_i$, as claimed.

\medskip

$(b)$ Suppose that $H \subseteq \bigcup_{i \in I} R_i$. First we consider the case when $H$ has finite rank. Let $x_1, \ldots, x_m$ be a basis of $H$; then 
$[x_1, \ldots, x_m] \in R_i$ for some $i \in I$, and it follows from part $(a)$ that $H \leq R_i$.

Now suppose that $H$ has infinite rank. Let $x_1, x_2, \ldots$ be a basis of $H$. Put $H_k=\langle x_1, \ldots, x_k \rangle$ for $k \geq 1$.
Then $H=\bigcup_{k=1}^{\infty} H_k$, and it follows from the finite rank case that for every $k \geq 1$, there is $i_k \in I$ such that $H_k \leq R_{i_k}$.
If none of the retracts $R_{i_k}  (k \geq 1$) contains $H$, then it is easy to see that there is a subsequence of indices $k_1 < k_2 < \ldots$ such that 
\[\bigcap_{j=1}^{\infty} R_{i_{k_j}} \lneq \bigcap_{j=2}^{\infty} R_{i_{k_j}} \lneq \bigcap_{j=3}^{\infty} R_{i_{k_j}} \lneq \ldots\] 
By Theorem~\ref{intersection of retracts}, each one of these intersections is a proper retract of $F_n$, and thus it has rank at most $n-1$. This is a contradiction with the well-known fact that every ascending sequence of subgroups of $F_n$ of bounded rank is stationary.

\end{proof}

\medskip

\begin{proof}[Proof of Theorem~\ref{main thm} $(i)$]
Let $x_1, x_2$ be a basis of $H$. Of course, we may assume that $H \cap R \neq \{1\}$.
We consider two cases.

\textbf{Case I:} $R$ does not contain a test element of $H$; by Howson's theorem, $R \cap H$ is finitely generated, and it follows from Theorem~\ref{Turner} that it can be covered by proper retracts of $H$: 
\[R \cap H \subseteq \bigcup_{i \in I} S_i,\] 
where $S_i$ is a proper retract of $H$ for every $i \in I$.
By Lemma~\ref{first step} $(b)$, $R \cap H \leq S_i$ for some $i \in I$. Since $S_i$ is a proper retract of $H$, it must be cyclic. Furthermore, since  retracts of free groups are isolated subgroups, it follows that $S_i \leq R$. 
Therefore, $H \cap R=S_i$.

\textbf{Case II:} $R$ contains a test element  $u(x_1,x_2)$ of $H$; set $S=\langle r(x_1), r(x_2) \rangle$ and $K=\langle H , S \rangle$.
By Theorem~\ref{pro-p basic} $(iii)$, there is a prime $p$ such that $u$ is a test element of $\widehat{H}_p$. Since $H$ is of rank two, the inclusion $H \hookrightarrow K$ extends to an isomorphism from
$\widehat{H}_p$ onto $\overline{H} \leq \widehat{K}_p$. Hence, $u$ is a test element of $\overline{H}$.  

Observe that $S$ is a retract of $K$. Hence, by Proposition~\ref{retracts in pro-p},  $\overline{S}$ is a free factor of $\widehat{K}_p$ and $\overline{S} \cap \overline{H}$ is a free factor of $\overline{H}$.
Since $u(x_1,x_2)=r(u(x_1,x_2))=u(r(x_1),r(x_2)) \in \overline{S} \cap \overline{H}$, it follows from Theorem~\ref{pro-p basic} $(i)$ that $\overline{S}=\widehat{K}_p$. Consequently, $S$ and $K$ both have rank two, and since $S$ is a retract of $K$, it follows that $S=K$. 
Therefore, $H \cap R=H$.
\end{proof}

\section{When $H$ is of rank $\geq 3$}
\label{higher rank}

An element $\mathbf{a}=(a_1, \ldots, a_n) \in \mathbb{Z}^n$ is called visible (or primitive) if it belongs to a basis of $\mathbb{Z}^n$, or equivalently, if
$\gcd(a_1, \ldots, a_n)=1$. Let $\pi_{ab}:F_n \to F_n^{ab}$ be the quotient homomorphism from $F_n$ onto its abelianization.
An element $w \in F_n$ is said to be visible if $\pi_{ab}(w)$ is visible in $F_n^{ab}$. 

The visible elements of $F_n$ are precisely the generators of cyclic retracts, that is,
$\langle w \rangle \leq F_n$ is a retract of $F_n$ if and only if $w$ is a visible element of $F_n$. 
Therefore, in the case of cyclic retracts, Bergman's question admits the following reformulation. 

\begin{question}
\label{reformulation}
Let $w$ be a visible element of $F_n$, and let $H$ be a finitely generated subgroup of $F_n$ that contains some non-trivial power of $w$.
Let $m$ be the smallest positive integer such that $w^m \in H$. Is $w^m$ a visible element of $H$?
\end{question}

The search for an $H$ for which the above question has a negative answer could be narrowed down to finite index subgroups. 
Indeed, suppose that $R$ is a retract of $F_n$ and $H$ is a finitely generated subgroup of $F_n$ such that $R \cap H$ is not a retract of $H$.
By Marshall Hall's theorem, $H$ is a free factor of some finite index subgroup $K$ of $F_n$. Furthermore, by the Kurosh subgroup theorem, $R \cap H$ is a free factor of $R \cap K$.
We claim that $R \cap K$ is not a retract of $K$; otherwise $R \cap H$ would also be a retract of $K$, and thus a retract of $H$, which contradicts our assumption.  

Our next result provides further guidance for finding the right $H$ and $w$.  We begin with some preliminaries. 

Let $\Gamma$ be an oriented $1$-dimensional CW complex (a directed graph) with one $0 \textrm{-cell}$ (vertex), denoted by $*$, and $n$ oriented $1$-cells (edges), $e_1, e_2, \ldots, e_n$. 
We think of $F_n$ as the fundamental group of $\Gamma$, and we let $x_i \in F_n$ stand for the homotopy class of the loop determined by the (directed) edge $e_i$ ($1 \leq i \leq n$).

For $1 \leq i \leq n$, let $\sigma_{x_i}: F_n \to \mathbb{Z}$ be the homomorphism defined by $\sigma_{x_i}(x_j)=\delta_{ij}$, where $\delta_{ij}$ is the Kronecker delta.
(Thus for $w \in F_n$, $\sigma_{x_i}(w)$ is the sum of the exponents of all occurrences of $x_i$ in $w$.) The first homology group of $\Gamma$ (with coefficients in $\mathbb{Z}$) is a free abelian group with basis $e_1, \ldots, e_n$; 
moreover, there is a  homomorphism $\sigma: F_n \to H_1(\Gamma)$, defined by $\sigma(w)=\sum_{i=1}^n \sigma_{x_i}(w)e_i$, that factors through an isomorphism from $F_n^{ab}$ onto $H_1(\Gamma)$.

For a given finite index subgroup $H$ of $F_n$, we denote by $\theta_H:F_n \to H^{ab}$ the transfer map. In the geometric context of covering spaces, $\theta_H$ can be described as follows. Let $f_H:(\tilde{\Gamma}_H, \tilde{*}_H) \to (\Gamma, *)$ be the pointed covering space corresponding to $H$.
Given an element $w=x_{i_1}^{\epsilon_1}x_{i_2}^{\epsilon_2} \ldots x_{i_k}^{\epsilon_k} \in F_n$ ($1 \leq i_j \leq n$ and $\epsilon_j \in \{-1, 1\}$ for every $1 \leq j \leq k$), let 
$p=e_{i_1}^{\epsilon_1}e_{i_2}^{\epsilon_2} \ldots e_{i_k}^{\epsilon_k}$ be the corresponding path in $\Gamma$, and for each vertex $v$ of $\tilde{\Gamma}_H$, let $\tilde{p}_v$ be the lift of $p$ with origin $v$.  
Let $c_1, c_2, \ldots, c_s$ be the cycles of the permutation induced by $w$ on the vertices of $\tilde{\Gamma}_H$. If $c_{t}=(v_1v_2 \ldots v_l)$ ($1 \leq t \leq s$), then 
$\tilde{p}_{v_1} \cdot\tilde{p}_{v_2} \cdot \ldots  \cdot\tilde{p}_{v_l}$ is a closed path in $\tilde{\Gamma}_H$; we denote by $h_{c_t}$ the corresponding homology class in $H_1(\tilde{\Gamma}_H)$. 
Then, upon identifying $H^{ab}$ with $H_1(\tilde{\Gamma}_H)$, we have $\theta_H(w)=h_{c_1}+h_{c_2}+ \ldots + h_{c_s}$. 

\begin{lem}
\label{transfer}
The transfer map $\theta_H:F_n \to H^{ab}$ sends visible elements to visible elements. 
\end{lem}

\begin{proof}
For each $1 \leq i \leq n$, let $\tilde{e}_i$ be the lift of $e_i$ to an oriented edge of $\tilde{\Gamma}_H$ with origin $\tilde{*}_H$.
Let $C_1(\tilde{\Gamma}_H)$ be the group of (cellular) $1$-chains of $\tilde{\Gamma}_H$, and let $\phi:H_1(\tilde{\Gamma}_H) \to H_1(\Gamma)$ be the restriction of the
homomorphism from $C_1(\tilde{\Gamma}_H)$ to $H_1(\Gamma)$ that sends $\tilde{e}_i$ to $e_i$ ($1 \leq i \leq n$) and maps all of the other edges of $\tilde{\Gamma}_H$ to $0$.

We claim that the following diagram commutes: 

\begin{displaymath}
    \xymatrix{ F_n \cong \pi_1(\Gamma) \ar[r]^{\theta_H} \ar[dr]^{\sigma} & H_1(\tilde{\Gamma}_H) \cong  H^{ab} \ar[d]^\phi \\
              & H_1(\Gamma) \cong F_n^{ab}}
\end{displaymath}

\medskip

Let $w=x_{i_1}^{\epsilon_1}x_{i_2}^{\epsilon_2} \ldots x_{i_k}^{\epsilon_k} \in F_n$, and let
$p=e_{i_1}^{\epsilon_1}e_{i_2}^{\epsilon_2} \ldots e_{i_k}^{\epsilon_k}$ be the corresponding path in $\Gamma$. 
For every vertex $v$ of $\tilde{\Gamma}_H$, let $\tilde{p}_v$ be the lift of $p$ to a path in $\tilde{\Gamma}_H$ with origin $v$. 
Then $\sigma(w)=\phi(\theta_H(w))$ is a consequence of the fact that for each $1 \leq j \leq k$, there is a unique vertex $v$ such that in $\tilde{p}_v$ the edge $e_{i_j}^{\epsilon_{i_j}}$ is lifted to $\tilde{e}_{i_j}^{\epsilon_{i_j}}$.

If $\theta_H(w)$ (for some $w \in F_n$) is not a visible element of $H^{ab}$, then we can write $\theta_H(w)=ma$ for some $a \in H^{ab}$ and $m \geq 2$.
Therefore, $\sigma(w)=\phi(\theta_H(w))=m\phi(a)$ is not a visible element of $F_n^{ab}$, and thus $w$ is not visible in $F_n$.
\end{proof}
\begin{pro}
Let $H$ be a finite index subgroup of $F_n$, and let $w \in F_n$ be a visible element. Suppose that $m$ is the smallest positive integer such that $w^m \in H$.
Then $w^m$ is a visible element of $H$ if one of the following holds: 
\begin{enumerate}[(a)]
\item $\langle w \rangle$ acts transitively on the (right) cosets of $H$;
\item $H$ is a subnormal subgroup of $F_n$. 
\end{enumerate}
\end{pro}
\begin{proof}
If $(a)$ holds, then $\theta_H(w)$ coincides with the image of $w^m$ in $H^{ab}$; hence, $w^m$ is visible in $H$ by Lemma~\ref{transfer}.

Suppose that $(b)$ holds; by induction, we may further assume that $H$ is a normal subgroup of $F_n$. 
Let $c_1, c_2, \ldots, c_s$ be the cycles of the permutation induced by $w$ on the vertices of $\tilde{\Gamma}_H$, and let $h_{c_1}, h_{c_2}, \ldots  ,h_{c_s}$ be the corresponding homology classes (defined as above).
Suppose that $c_1$ is the cycle containing $\tilde{*}_H$; then $h_{c_1}$ is the image of $w^m$ in $H^{ab}$.

The group of deck transformations, $G/H$, of $\tilde{\Gamma}_H$ acts transitively on the set $\{h_{c_1}, h_{c_2}, \ldots  ,h_{c_s}\} \subseteq H_1(\tilde{\Gamma}_H)$; for each $2 \leq j \leq s$,  fix $g_j \in G/H$ such that $h_{c_1}^{g_j}=h_{c_j}$.
If $h_{c_1}=ma$  for some $a \in H^{ab}$ and $m \geq 1$, then 
\[\theta_H(w)=h_{c_1} + h_{c_2} + \ldots  + h_{c_s}=h_1+h_1^{g_2}+ \ldots + h_1^{g_s}=m(a+a^{g_1}+ \ldots+a^{g_s}).\]
Since, by Lemma~\ref{transfer}, $\theta_H(w)$ is visible in $H^{ab}$, it follows that $m=1$. Therefore, $h_{c_1}$ is visible in $H^{ab}$ and $w^m$ is visible in $H$.  
\end{proof}

\medskip 

For a while, our discussion will be restricted to the free group of rank two. Accordingly, we let $\Gamma$ denote the CW complex with one $0$-cell and two (oriented) 1-cells, $e$ and $f$.
We think of $F_2$ as the fundamental group of $\Gamma$, and we let $x$ and $y$ stand for the homotopy classes of the loops determined by $e$ and $f$, respectively. 
  
For $m \geq 2$, let $\Gamma_m$ be a CW complex with $m$ $0$-cells, $v_0, v_1, \ldots, v_{m-1}$, and $2m$ (oriented) $1$-cells, $e_i, f_i$ ($0 \leq i \leq m-1$). 
The origin of both $e_i$ and $f_i$ is the vertex $v_{i}$; the terminus of $e_i$ is $v_{m-i}$ and the terminus of $f_i$ is $v_{i+1}$, where the indices are taken modulo $m$ (see Figure~\ref{figure}).
Let $h_m:\Gamma_m \to \Gamma$ be the (graph) map defined by $h_m(e_i)=e$ and $h_m(f_i)=f$ for every $1 \leq i \leq m-1$.
It is straightforward to verify that $h_m$ is a covering map. 

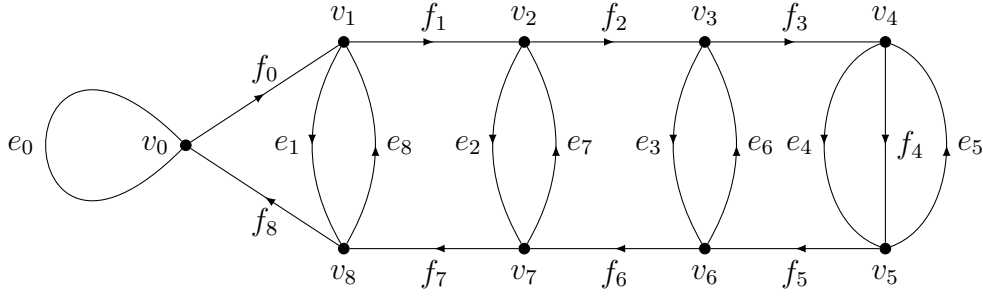
\begin{figure}[t]
\[\begin{tikzpicture}[x=.3 cm, y=.55 cm, 
    every edge/.style={
        draw,
        postaction={decorate,
                    decoration={markings,mark=at position 0.5 with {\arrow{latex}}}
                   }
        }
]
	\vertex[fill] (v_0) at (-7,0)[label=left:$v_0$]{};
	\vertex[fill] (v_1) at (0,2.5)[label=above:$v_1$]{};
        \vertex[fill] (v_2) at (8,2.5)[label=above:$v_2$]{};		
        \vertex[fill] (v_3) at (16,2.5)[label=above:$v_3$]{};			
        \vertex[fill] (v_4) at (24,2.5)[label=above:$v_4$]{};		
        \vertex[fill] (v_8) at (0,-2.5)[label=below:$v_8$]{};
        \vertex[fill] (v_7) at (8,-2.5)[label=below:$v_7$]{};		
        \vertex[fill] (v_6) at (16,-2.5)[label=below:$v_6$]{};			
        \vertex[fill] (v_5) at (24,-2.5)[label=below:$v_5$]{};

\path

        (v_0) edge node[above]{$f_0$} (v_1)
	(v_1) edge node[above]{$f_1$} (v_2)
        (v_2) edge node[above]{$f_2$} (v_3)
        (v_3) edge node[above]{$f_3$} (v_4)
        (v_4) edge node[right]{$f_4$} (v_5)
        (v_5) edge node[below]{$f_5$} (v_6)
        (v_6) edge node[below]{$f_6$} (v_7)
        (v_7) edge node[below]{$f_7$} (v_8)
        (v_8) edge node[below]{$f_8$} (v_0)
        (v_1) edge[bend right] node[left]{$e_1$} (v_8)
        (v_2) edge[bend right] node[left]{$e_2$} (v_7)
        (v_3) edge[bend right] node[left]{$e_3$} (v_6)
        (v_4) edge[bend right=70] node[left]{$e_4$} (v_5)        
        (v_8) edge[bend right] node[right]{$e_8$} (v_1)
        (v_7) edge[bend right] node[right]{$e_7$} (v_2)
        (v_6) edge[bend right] node[right]{$e_6$} (v_3)
        (v_5) edge[bend right=70] node[right]{$e_5$} (v_4)
;
\draw (v_0) to [out=225,in=135,looseness=80] (v_0);
\draw (-14.3,0) node{$e_0$};
\end{tikzpicture}\]
\caption{The graph $\Gamma_9$.\label{figure}}
\end{figure}

Consider the subgroup $H_m=(h_m)_*(\pi_1(\Gamma_m, v_0))$ of $F_2$. The edges $f_0, f_1, \ldots, f_{m-2}$ form a spanning subtree of $\Gamma_m$; the basis of $H_m$ determined (in the standard way) by this tree is 
$x, t_1=yxy^{-(m-1)}, t_2=y^2xy^{-(m-2)}, \ldots, t_{m-1}=y^{m-1}xy^{-1}, y^m$. 

It will be convenient to have another way of describing the subgroup $H_m$. Let $\psi_m:F_2 \to D_m = \langle  t, s \mid t^2=1, s^m=1, tst=s^{-1} \rangle$ be the homomorphism defined by $\psi_m(x)=t$ and $\psi_m(y)=s$.
We claim that $H_m=\psi_m^{-1}(\langle t \rangle)$. Indeed, it is easy to check that $H_m \leq \psi_m^{-1}(\langle t \rangle)$, and since both subgroups have index $m$ in $F_2$, they must coincide.

For $k \geq 1$, set $w_k=x[x,y]^k \in F_2$.  (Here, $[x,y]=xyx^{-1}y^{-1}$.)
Each $w_k$ is a visible element in $F_2$. Since $\psi_m(w_k)=t[t,s]^k=ts^{-2k}$, we have that $w_k \in H_m$ if and only if $m \mid 2k$. Moreover,
$\psi_m(w_k^2)=1$, and thus $w_k^2 \in H_m$ for every $k \geq 1$.
 
\begin{lem}
\label{example}
Let $m=2k+1$ with $k \geq 1$. Then $w_k \notin H_m$ and 
\[w_k^2=x^2t_{m-1}^{-1}t_{m-2}t_{m-3}^{-1}t_{m-4} \ldots t_2^{-1}t_1t_{m-1}t_{m-2}^{-1}t_{m-3}t_{m-4}^{-1} \ldots t_2t_1^{-1} \in H_m. \]
In particular, $w_k^2$ is not a visible element of $H_m$. 
\end{lem}
\begin{proof}
Since $m$ does not divide $2k$, it follows from the discussion before the lemma that $w_k \notin H_m$. Consider the path
\[p_k=ee \underbrace{(fe^{-1}f^{-1}e)(fe^{-1}f^{-1}e) \ldots (fe^{-1}f^{-1}e)}_{k \text{ times}}e\underbrace{(fe^{-1}f^{-1}e) \ldots (fe^{-1}f^{-1}e)}_{(k-1) \text{ times}}fe^{-1}f^{-1}\]
in $\Gamma$; it determines the element $w_k^2 \in F_2$.

For each $0 \leq i \leq m-3$, the path $fe^{-1}f^{-1}e$ in $\Gamma$ lifts in $\Gamma_m$ to the path $f_i e_{m-(i+1)}^{-1} f_{m-(i+2)}^{-1} e_{m-(i+2)}$  with origin $v_i$ and terminus $v_{i+2}$.
Furthermore, the path $fe^{-1}f^{-1}$ lifts to the path $f_{m-2} e_{1}^{-1} f_{0}^{-1}$ with origin $v_{m-2}$ and terminus $v_0$.

It follows now easily that $p_k$ lifts to the following loop at $v_0$ in $\Gamma_m$:
\begin{align*}
\tilde{p}_k=&e_0e_0 (f_0e_{m-1}^{-1}f_{m-2}^{-1}e_{m-2})(f_{2}e_{m-3}^{-1}f_{m-4}^{-1}e_{m-4}) \ldots (f_{m-3}e_{2}^{-1}f_1^{-1}e_1) \\ 
 &e_{m-1} (f_1 e_{m-2}^{-1}f_{m-3}^{-1}e_{m-3}) \ldots (f_{m-4} e_{3}^{-1}f_{2}^{-1}e_2)f_{m-2}e_1^{-1}f_0^{-1}.
\end{align*}
Finally, the expression for $w_k^2$ in terms of the basis of $H_m$ can be read from $\tilde{p}_k$.
\end{proof}

\begin{example}
For a simple example that provides a negative answer to Bergman's question, take $K=\langle x, t_1=yxy^{-2}, t_2=y^2xy^{-1} \rangle \leq F_2$. 
Then $w_1=x[x,y]$ is a visible element of $F_2$, and thus $R=\langle w_1 \rangle$ is a retract of $F_2$;
since $w_1^2=x^2[t_2^{-1}, t_1]$ is not a visible element of $K$, it follows that $K \cap R=\langle  w_1^2 \rangle$ is not a retract of $K$.
\end{example}

\medskip

\begin{proof}[Proof of Theorem~\ref{main thm} $(ii)$]
For $m \geq 3$, we set $L_m=H_{m-1}$ if $m$ is even; otherwise we define $L_m$ to be the subgroup of $H_m$ generated by $x, t_0, t_1, \ldots, t_{m-1}$.
Observe that $\textrm{rk}(L_m)=m$ for every $m \geq 3$.
 
For $m \geq 1$, let $R_m=\langle w_k \rangle \leq F_2$, where $k=\lfloor \frac{m-1}{2} \rfloor$. Then $R_m$ is a retract of $F_2$, and
it follows from Lemma~\ref{example} that $R_m \cap L_m$ is not a retract of $L_m$. This proves Theorem~\ref{main thm} $(ii)$ in the case when $n=2$.

For the general case,  we write $F_n=F_2 * F_{n-2}$ ($F_1$ stands for an infinite cyclic group), and we consider $L_m$ and $R_m$ as subgroups of the first factor of the free product decomposition of $F_n$.
Given any retract $S$ of the second factor, we have that $R_m*S$ is a retract of $F_n$. However, $(R_m*S) \cap L_m=R_m \cap L_m$ is not a retract of $L_m$.
\end{proof}

In \cite{SnoTan4}, the following  conjecture was made

\begin{conjecture}
\label{conj-test elements}
Let $H$ be a finitely generated subgroup of $F_n$ that is not contained in a proper retract of $F_n$. 
Then every test element of $H$ is a test element of $F_n$.
\end{conjecture}

As a consequence of Theorem~\ref{main thm}, we obtain the following 

\begin{thm}
Conjecture~\ref{conj-test elements} does not hold. 
\end{thm}
\begin{proof}
Choose a subgroup $H$ of $F_n$ of rank three for which there exists a retract $R$ of $F_n$ such that $H \cap R$ is not a retract of $H$.
Let $S$ be the intersection of all retracts of $F_n$ that contain $H$. By Theorem~\ref{intersection of retracts}, $S$ and $R \cap S$ are retracts of $F_n$.
Furthermore, $R \cap S$ is a proper retract of $S$, $H \cap (R \cap S)=H \cap R$ is not a retract of $H$, and $H$ is not contained in a proper retract of $S$.
Hence, after replacing $F_n$ by $S$ and replacing $R$ by $R \cap S$, we may assume that $H$ is not contained in a proper retract of $F_n$.

If $H \cap R$ does not contain a test element of $H$, then by Theorem~\ref{Turner}, there are proper retracts $T_i, i \in I,$ of $H$ such that 
$H \cap R \subseteq \bigcup_{i \in I} T_i$. By Lemma~\ref{first step} $(b)$, $H \cap R \subseteq T_i$ for some $i \in I$.
Since $T_i$ is a proper retract of $H$, it has rank at most two, and it follows from Theorem~\ref{main thm} $(i)$ that  $T_i \cap R$ is a retract of $T_i$.
This implies that $H \cap R=T_i \cap R$ is a retract of $H$, a contradiction. 

Therefore, $H \cap R$ contains a test element of $H$. However, by Theorem~\ref{Turner}, no element contained in $R$ is a test element of $F_n$. 
\end{proof}

\section{The Dicks-Ventura Conjecture}
\label{Dicks-Ventura}

Throughout this section, $p$ denotes a fixed prime.  The pro-$p$ topology of $F_n$ is the coarsest topology with respect to which $F_n$ is a topological group and every homomorphism from $F_n$ into a finite $p$-group is continuous. The normal subgroups of $F_n$ of index a finite power of $p$ form a base for the neighborhoods of the identity for the pro-$p$ topology.

Given a subgroup $H$ of $F_n$, we denote by $\textrm{cl}(H)$ the closure of $H$ in the pro-$p$ topology of $F_n$ (the notation $\overline{H}$ continues to be used for the closure of $H$ in $\widehat{F}_{n,p}$).

\begin{lem}
\label{lem D-V}
Let $R$ be a retract of $F_n$, and let $H$ be a finitely generated subgroup of $F_n$. Then the following holds:
\begin{enumerate}[(a)]
\item If $H$ is closed in the pro-$p$ topology of $F_n$, then either $H \leq R$ or $\textrm{rk}(H \cap R) < \textrm{rk}(H)$.
\item If $H$ is not contained in $R$, then $\textrm{rk}(\textrm{cl}(H) \cap R) < \textrm{rk}(H)$.
\end{enumerate}
\end{lem}
\begin{proof}
Suppose that $H$ is closed in the pro-$p$ topology of $F_n$ and that $R$ does not contain $H$. Since $R$ is a retract of $F_n$, it follows from
\cite[Lemma~3.1.5]{Ribes1} that it is closed in the pro-$p$ topology of $F_n$. Hence, $R \cap H$ is also closed in the pro-$p$ topology. 
Furthermore, by \cite[Proposition~2.3]{Ribes-Zalesski2014} (see also \cite[Proposition~13.1.4]{Ribes-Book}), $\overline{H \cap R}=\overline{H} \cap \overline{R}$ in $\widehat{F}_{n,p}$, and
it follows from \cite[Lemma~5.3 $(2)$ and Corollary~5.8 $(b)$]{NH-RZ} that
\[\textrm{rk}(H \cap R)=\textrm{rk}(\overline{H \cap R})=\textrm{rk}(\overline{H} \cap \overline{R}).\]
By Proposition~\ref{retracts in pro-p}, $\overline{R}$ is a free factor of $\widehat{F}_{n,p}$ and $\overline{H} \cap \overline{R}$ is a free factor of 
$\overline{H}$. Moreover, by \cite[Lemma~5.4]{NH-RZ}, $F_n \cap \overline{H \cap R}=H \cap R < H=F_n \cap \overline{H}$. It follows that $\overline{H} \cap \overline{R}$ is in fact a proper free factor of $\overline{H}$. Hence, 
\[\textrm{rk}(H \cap R)=\textrm{rk}(\overline{H} \cap \overline{R}) < \textrm{rk}(\overline{H})=\textrm{rk}(H).\]

Now $(b)$ follows from $(a)$ and the fact that $\textrm{rk}(cl(H)) \leq \textrm{rk}(H)$ (see \cite[Proposition~11.3.1]{Ribes-Book}).
\end{proof}

\begin{proof}[Proof of Theorem~\ref{D-V}]
We may suppose that $H$ is not contained in $R$.  By Lemma~\ref{lem D-V} $(b)$, 
\[\textrm{rk}(\textrm{cl}(H) \cap R) < \textrm{rk}(H) \leq 3.\]
By the Hanna Neumann conjecture (see \cite{Friedman}, \cite{Mineyev} and \cite{Jaikin-Zapirain}; in fact, here we only need the special case, first proved in \cite{Tardos}, when one of the intersecting subgroups is of rank two), we have 
\[\textrm{rk}(H \cap R)=\textrm{rk}(H \cap (\textrm{cl}(H) \cap R)) \leq \textrm{rk}(H).\]
\end{proof}

For completeness, we prove Corollary~\ref{cor D-V}, although it follows from Theorem~\ref{D-V} by a well-known argument.

\begin{proof}[Proof of Corollary~\ref{cor D-V}]
We call a subgroup $H$ of $F_n$ $3$-inert if $\textrm{rk}(K \cap H) \leq \textrm{rk}(K)$ for every subgroup $K$ of $F_n$ with $\textrm{rk}(K) \leq 3$. 
We need to show that $\textrm{Fix}(S)$ is a $3$-inert subgroup of $F_n$ for every family $S$ of endomorphisms of $F_n$. Since the property of being $3$-inert is  closed under intersections,
we may assume that $S$ consists of a single endomorphism $\varphi$ of $F_n$.

By \cite[Theorem~1]{Turner}, $\varphi^{\infty}(F_n)=\bigcap_{k=1}^{\infty} \phi^k(F_n)$ is a retract of $F_n$. Moreover, by \cite[Theorem~1]{Imrich-Turner}, $\varphi(\varphi^{\infty}(F_n))=\varphi^{\infty}(F_n)$ and $\varphi^{\infty}=\varphi_{|\varphi^{\infty}(F_n)}:\varphi^{\infty}(F_n) \to \varphi^{\infty}(F_n)$
is an automorphism. It follows from the main theorem of \cite{Dicks} that $\textrm{Fix}(\varphi)=\textrm{Fix}(\varphi^{\infty})$ is $3$-inert in $\varphi^{\infty}(F_n)$. By Theorem~\ref{D-V}, $\varphi^{\infty}(F_n)$ is $3$-inert in $F_n$.
Since the property of being $3$-inert is transitive, it follows that $\textrm{Fix}(\varphi)$ is $3$-inert in $F_n$. 
\end{proof}

We end the paper with some further evidence for the Dicks-Ventura conjecture.

\begin{pro}
Let $R$ be a retract of $F_n$, and let $H$ be a finitely generated subgroup of $F_n$ with $[\textrm{cl}(H):H] < \infty$.
Then $\textrm{rk}(H \cap R) \leq \textrm{rk}(H)$.
\end{pro}
\begin{proof}
Set $m=[\textrm{cl}(H):H]$ and $k=[\textrm{cl}(H) \cap R: H \cap R]$; observe that $k \leq m$. It follows from Lemma~\ref{lem D-V} $(a)$ and the Schreier formula that
\[\textrm{rk}(H \cap R) = k(\textrm{rk}(\textrm{cl}(H) \cap R)-1)+1 \leq m(\textrm{rk}(\textrm{cl}(H))-1)+1=\textrm{rk}(H).\]
\end{proof}

\ackn 
The first author acknowledges support from the Alexander von Humboldt Foundation, CAPES (grant 88881.145624/2017-01) and FAPERJ. 
The third  author is partially supported by CNPq.
The authors thank the Heinrich Heine University in D\"usseldorf for its hospitality.
\bibliographystyle{plain}

\end{document}